\documentclass{amsart}

\usepackage{amssymb, graphicx}
\usepackage{stmaryrd}
\usepackage[all]{xy}
\usepackage{xcolor}
\usepackage{hyperref}

\newtheorem{thm}{Theorem}[section]

\newtheorem{prop}[thm]{Proposition}

\theoremstyle{definition}
\newtheorem{defn}[thm]{Definition}

\newtheorem{exmp}[thm]{Example}

\newtheorem{ques}[thm]{Question}    

\newtheorem{rem}[thm]{Remark}          

\newtheorem*{ack}{Acknowledgment}      
\newtheorem{notation}[thm]{Notation}   
  
\newtheorem{defn-thm}[thm]{Definition--Theorem}  
\newtheorem{defn-lem}[thm]{Definition--Lemma}  

\theoremstyle{remark}

\newcommand{\rank}[0]{\operatorname{rank}}

\newcommand{\coker}[0]{\operatorname{coker}}

\newcommand{\sing}[0]{\operatorname{Sing}}

\newcommand{\pf}[0]{\operatorname{Pf}}

\newcommand{\uc}[0]{\operatorname{uc}}

\def\loccoh#1.#2.#3.#4.{H^{#1}_{#2}(#3,#4)}

\DeclareMathAlphabet{\mathchanc}{OT1}{pzc}%
                                {m}{it}

\newcommand\scalemath[2]{\scalebox{#1}{\mbox{\ensuremath{\displaystyle #2}}}}

\begin{document}
\bibliographystyle{amsalpha}


\title[An explicit matrix factorization of small cubics]{An explicit matrix factorization of cubic hypersurfaces of small dimension}
\author{Yeongrak Kim}
\email{kim@math.uni-sb.de}

\author{Frank-Olaf Schreyer}
\email{schreyer@math.uni-sb.de}
\address{F. Mathematik und Informatik, Universit{\"a}t des Saarlandes, Campus E2.4, D-66123 Saarbr{\"u}cken, Germany}

\keywords{Matrix factorization, Ulrich module, cubic hypersurface, Shamash's construction, spinor variety, Cartan cubic}

\begin{abstract}
In this paper, we compute an explicit matrix factorization of a rank 9 Ulrich sheaf on a general cubic hypersurface of dimension at most 7, whose existence was proved by Manivel. Instead of using invariant theory, we use Shamash's construction with a cone over the spinor variety. We also describe an algebro-geometric interpretation of our matrix factorization which connects the spinor tenfold and the Cartan cubic.
\end{abstract}
\maketitle
\section{Introduction}
Let $S$ be a polynomial ring, and let $f \in S$ be a polynomial. A pair of matrices $(A,B)$ with entries in $S$ is called a matrix factorization of $f$  if $AB=BA=f \cdot Id$, where $Id$ is an identity matrix (of some size). It was introduced by Eisenbud \cite{Eis80} in the context of commutative algebra to study free resolutions over the hypersurface ring $R=S/(f)$. Several applications of matrix factorizations were discovered recently, for instance, a strong connection between the string theory as categories of $D$-branes for Landau-Ginzburg B-models \cite{KL04, Orl04}. 

In commutative algebra, there is an important connection between matrix factorizations and Cohen-Macaulay modules. Among them, we are particularly interested in a matrix factorization $(A,B)$ of a nonzero homogeneous polynomial $f \in S$ where every entry of $A$ is linear. When it exists, an $R$-module $M:=\coker \left( \oplus S(-1) \stackrel{A} \to \oplus S \right)$ has a completely linear $S$-resolution of length $1$. Such a module has a maximal number of generators (in degree $0$) it can have. It is called a maximally generated maximal Cohen-Macaulay module, or an \emph{Ulrich module}, to memorize a pioneering work of Ulrich \cite{Ulr84} and follow-ups. Eisenbud and Schreyer \cite{ESW03} introduced the notion of an Ulrich sheaf which is an analogous object defined in a geometric setting. Several remarkable applications, including representations of the Cayley-Chow form and the cone of cohomology tables, emphasize the importance of the study of Ulrich sheaves. 

In general, finding an Ulrich module supported on a given variety is not simple. Fortunately, it is well known that there is an Ulrich module supported on the hypersurface $V(f)$ for any homogeneous polynomial $f \in S$. Backelin and Herzog showed this existence by construction \cite{BH87} using Childs' analysis on the Roby-Clifford algebra \cite{Chi78}. However, their construction only provides an Ulrich module of huge rank, which seems to be very far away from the smallest possible rank in many cases. The smallest possible rank, called the \emph{Ulrich complexity} \cite{BES17}, contains a number of open problems. When $f$ defines a smooth  quadric hypersurface in $\mathbb{P}^n$, we know the exact answer: the only indecomposable Ulrich modules are the spinor modules (there are $1$ or $2$, depends on the parity of $n$) of rank $r=2^{\lfloor (n-2)/2 \rfloor}$ \cite{BEH87}. Except for smooth quadrics, only a few cases are explicitly understood, when the degree and the number of variables are very small \cite{Bea00}.

Hence, it is natural to ask the Ulrich complexity of a hypersurface cut out by a (very) general homogeneous cubic polynomial $f$ defined in $(n+1)$ variables $x_0, \cdots, x_n$. Let us recall the known cases over $\mathbb{C}$. When $f$ defines a curve or a surface, then it is classically well-known that $f$ is linearly determinantal, \emph{i.e.}, $f = \det A$ for some $3 \times 3$ linear matrix $A$. It is clear that such $A$ induces a matrix factorization of $f$ \cite[Section 5]{Eis80}, and thus presents an Ulrich module of rank $1$. When $X=V(f)$ is a general cubic threefold, then it is no more linearly determinantal but linearly Pfaffian, \emph{i.e.}, $f = \pf A$ for some $6 \times 6$ skew-symmetric linear matrix $A$. Similar as above, $A$ presents an Ulrich module of rank $2$, which gives the Ulrich complexity of $X$. 

Suppose that a smooth cubic fourfold $X=V(f)$ supports an Ulrich module of rank $2$. Then it is linearly Pfaffian, and such a cubic fourfold always contains a del Pezzo surface of degree $5$ \cite{Bea00}. In particular, such cubic fourfolds form a divisor in the space of cubic fourfolds, and hence the Ulrich complexity of a (very) general cubic fourfold is at least $3$. Indeed, the Ulrich complexity of a general cubic form $f$ in $n+1$ variables is not exactly known when $n \ge 5$.

Very recently, Manivel showed the existence of rank $9$ Ulrich sheaves on a general cubic hypersurface of small dimensions:
\begin{thm}[\cite{IM14, Man19}, see also Theorem \ref{Prop:Rk9UlrichExistence}]
There is an Ulrich sheaf of rank $9$ on a general cubic hypersurface of dimension at most $7$.
\end{thm}
Using invariant theory, he found an $E_6$-equivariant linear map whose cokernel is supported on the Cartan cubic hypersurface in $\mathbb{P}^{26}$. In particular, this gives a rank 9 Ulrich sheaf $\mathcal E$ on the Cartan cubic. Since a general cubic hypersurface of dimension at most $7$ can be obtained as a general linear section of the Cartan cubic \cite[Proposition 2.2]{IM14}, the restriction of $\mathcal E$ will be an Ulrich sheaf of the same rank on a general linear section \cite[Lemma 2.4]{CH12}. Although the construction is clear, the paper does not focus on an explicit description of the matrix factorization he obtained. 

Hence, the purpose of this note is to reprove Manivel's result by addressing an explicit matrix factorization which provides an Ulrich sheaf of rank $9$. We use Shamash's construction to compute such a matrix factorization instead of invariant theory. Surprisingly, two different ideas intersect on geometry of the Cartan cubic, since it is possible to recover the Cartan cubic from the spinor variety in a clear way (= Theorem \ref{prop:RecoverCartanCubic}). Consequently, a matrix factorization of a general cubic hypersurface of dimension at most $7$ can be obtained as a restriction of a matrix factorization of the Cartan cubic.

The structure of the paper is as follows. In Section \ref{Sect:Preliminaries}, we recall some basic notions and helpful results on Ulrich sheaves and Shamash's construction. In Section \ref{Sect:MFCubic}, we compute an explicit matrix factorization of a general cubic hypersurface of dimension at most $7$, which corresponds to an Ulrich module of rank $9$ as an application of Shamash's construction. And then, we reconstruct the Cartan cubic from the spinor tenfold $\mathcal S_{10} \subset \mathbb{P}^{15}$ to observe a connection between two different constructions of Ulrich modules. In fact, the Hessian matrix of the Cartan cubic form induces a matrix factorization of the Cartan cubic (see also \cite{Abu18}), which is compatible with our computation. 

\section{Preliminaries}\label{Sect:Preliminaries}
We briefly recall some preliminaries which appear in the whole paper. We work over the field $\mathbb{C}$ of complex numbers to fit with the classical setting, however, most of computations can be done in a similar way over an algebraically closed field of characteristic $\neq 2, 3$. 

\begin{notation}
Throughout the paper, we use the following notations. 
\[
\begin{array}{rl}
k  = \mathbb{C} & \text{the complex number field}; \\
X \subset \mathbb{P}^n & \text{a connected projective variety, mostly a cubic hypersurface}; \\ 
S=k[x_0, \cdots, x_n] & \text{the homogeneous coordinate ring of } \mathbb{P}^n; \\
S_X=S/I_X & \text{the homogeneous coordinate ring of } X; \\
\end{array}
\]
\end{notation}

\begin{defn}[See also {\cite[Proposition 2.1]{ESW03}}]
A coherent sheaf $\mathcal E$ supported on $X$ is called an \emph{Ulrich sheaf} if its twisted section module $\Gamma_* (\mathcal E)$ is an \emph{Ulrich module}, that is, the minimal $S$-free resolution 
\[
F_{\bullet} : 0 \to F_c \to \cdots \to F_1 \to F_0 \to \Gamma_* (\mathcal E) \to 0
\]
of $\Gamma_{*}(\mathcal E)$ is completely linear, in the sense that $F_i \simeq \oplus S(-i)$ is generated in degree $i$ for every $0 \le i \le c=n- \dim X$. 
\end{defn}
In particular, an Ulrich module is a maximal Cohen-Macaulay module which has a completely linear $S$-resolution. Since the sheaf associated to an Ulrich module is an Ulrich sheaf, we will not distinguish these notions. 

Recall that the Ulrich complexity is defined to be the smallest rank of Ulrich sheaves on $X$, denoted by $\uc(X)$. The most important question, suggested first by several commutative algebraist, and whose positive answer is nowadays called a conjecture of Eisenbud and Schreyer, is:

\begin{ques}[\cite{ESW03, ES11}]
Does every $X$ support an Ulrich sheaf, \emph{i.e.}, $\uc(X)<\infty$?
\end{ques}

When $X$ carries an Ulrich sheaf, then the cone of cohomology tables for $X$ is identical to the cone of the cohomology tables for the projective space of the same dimension $\mathbb{P}^{\dim X}$, regardless of the rank of an Ulrich sheaf we chose \cite[Theorem 4.2]{ES11}. However, in practice, we are much interested in Ulrich sheaves of smaller rank as possible. 

It is worthwhile to recall how matrix factorizations and Ulrich sheaves are related before proceed. Let $X=V(f) \subset \mathbb{P}^n$ be a hypersurface cut out by a nonzero homogeneous form $f$ of degree $d$. It is well-known that a matrix factorization $(A,B)$ of $f$ induces a maximal Cohen-Macaulay module supported on $X$ by $\coker A$. Conversely, if we have a maximal Cohen-Macaulay $S_X=S/(f)$-module, one has a matrix $A$ by reading off its minimal free resolution of length 1. Such a matrix $A$ forms a part of a matrix factorization of $f$, \emph{i.e.}, there is a unique matrix $B$ such that $AB=BA=f \cdot Id$. Indeed, there is a bijection between the isomorphism classes of maximal Cohen-Macaulay modules and the equivalence classes of matrix factorizations of $f$ \cite[Section 5, 6]{Eis80}. In particular, an Ulrich module on the hypersurface ring $S_X$ gives a matrix factorization $(A,B)$ of $f$ where $A$ is its presentation matrix whose entries are linear forms, and vice versa. Since $A$ determines $B$ (and $B$ also determines $A$), we sometimes call $A$ is a matrix factorization of $f$ (or, of $X$) for simplicity. We refer to \cite{Eis80} for more details on the matrix factorization. 

Questions on the Ulrich complexity are much more mysterious even for hypersurfaces. The first open question on the Ulrich complexity of a general hypersurface is:
\begin{ques}
What is the Ulrich complexity of a general cubic fourfold in $\mathbb{P}^5$?
\end{ques}

Note that the construction in \cite{BH87} provides an upper bound.  Since the Chow rank of a general cubic form $f$ in $6$ variables is $4$ \cite[Corollary 5.2]{Abo14}, which means, $f$ can be written as a sum of $4$ completely decomposable forms
\[
f = \ell_{1,1} \ell_{1,2} \ell_{1,3}+ \ell_{2,1} \ell_{2,2} \ell_{2,3}+ \ell_{3,1} \ell_{3,2} \ell_{3,3}+ \ell_{4,1} \ell_{4,2} \ell_{4,3}
\]
where each $\ell_{i,j}$ is linear. Following the arguments in \cite{Chi78, BH87, BES17}, there are $27 \times 27$ linear matrices $A_1, A_2, A_3$ such that $A_1 A_2 A_3 = A_2 A_3 A_1 = A_3 A_1 A_2 = f \cdot Id_{27}$. In particular, $(A_1, A_2 A_3)$ is a matrix factorization of $f$, and hence a general cubic fourfold always carries an Ulrich sheaf of rank $9 = 27/3$. 

Unfortunately, this construction only provides an upper bound which is quite far from the Ulrich complexity in so many cases. For instance, the Chow rank of a general cubic form in $7$ variables jumps to $5$. In the case, the above argument yields a matrix factorization by a $81 \times 81$ linear matrix, which defines an Ulrich sheaf of rank $27$. On the other hand, Manivel's observation implies that there is an Ulrich sheaf of rank $9$ on a general cubic sevenfold \cite[Corollary 2.3]{Man19}, and hence $\uc (X) \le 9 < 27$ when $X$ is a general cubic hypersurface of dimension at most $7$.

Let us describe a little more details on the Ulrich complexity of a very general cubic fourfold. Let $X$ be a very general cubic fourfold in $\mathbb{P}^5$, and let $\mathcal E$ be an Ulrich sheaf on $X$. Note that $\mathcal E$ is Ulrich if and only if $H^i (X, \mathcal E(-j))=0$ for every $i, 1 \le j \le 4 = \dim X$ \cite[Section 2]{ESW03}. In particular, $\mathcal E$ has no intermediate cohomology. Since $X$ is smooth, any coherent sheaf without intermediate cohomology on $X$ is locally free. Indeed, any Ulrich sheaf on $X$ is locally free, and hence we may read off its cohomology from the Riemann-Roch formula. Also note that intersection theory on $X$ is determined by multiples (by a rational number) of codimension $i$ cycles $H^i$, where $H \subset X$ denotes the general hyperplane section of $X$ so that $H^4=3$. From the short exact sequence
\[
0 \to \mathcal T_X \to \mathcal T_{\mathbb{P}^5}|_X \to \mathcal N_{X/\mathbb{P}^5} = \mathcal O_X (3) \to 0,
\]
we deduce the Riemann-Roch formula for $X$
\[
\chi(\mathcal E) = \left[ ch(\mathcal E) \cdot \left(1 + \frac{3}{2}H + \frac{5}{4}H^2 + \frac{3}{4} H^3 + \frac{1}{3}H^4 \right) \right]_4.
\]

For simplicity, we compute the Chern classes of $\mathcal F = \mathcal E(-1)$ instead of $\mathcal E$. Let us denote $c_i = c_i (\mathcal F), r = \rank (\mathcal E) = \rank (\mathcal F)$. Since $\chi(\mathcal F) = \chi(\mathcal F(-1))= \chi(\mathcal F(-2))= \chi(\mathcal F(-3))=0$ \cite[Proposition 2.1]{ESW03}, we conclude that
\[
c_1 = 0,\ c_2  = \frac{r}{3} H^2,\ c_3 = 0, \ c_4 = \left( \frac{1}{18}r^2 - \frac{1}{2}r \right) H^4 = \frac{1}{6}r^2 - \frac{3}{2} r. 
\]

\begin{prop}\label{prop:lowerboundUCbyRiemannRoch}
Let $\mathcal E$ be an Ulrich sheaf of rank $r$ on a very general cubic fourfold $X \subset \mathbb{P}^5$. Then $r$ is divisible by $3$ and $r \ge 6$.
\end{prop}

\begin{proof}
Since $c_4 (\mathcal E(-1)) = \frac{1}{6}r^2 - \frac{3}{2}r$ is an integer, $r$ must be divisible by $3$. Suppose that there is an Ulrich sheaf $\mathcal E$ of rank $3$. The cohomology condition forces that the value $\left(\frac{1}{6} r^2 - \frac{3}{2}r \right) = -3$ must be different from zero, however, $c_4$ must be zero since $\mathcal E$ is of rank $3$. Hence, we conclude that a rank 3 Ulrich bundle cannot exist on $X$.
\end{proof}

To combine both observations, the Ulrich complexity of a very general cubic fourfold is either $6$ or $9$. Unfortunately, we do not know yet whether there is an Ulrich sheaf of rank $6$ on a very general cubic fourfold, or not. On the other hand, it is much easier to find an Ulrich sheaf of rank 9 on a very general cubic fourfold (and hypersurfaces of dimension $\le 7$). Therefore, it is worthwhile to study Ulrich sheaves of rank $9$ on a (very) general cubic hypersurface of small dimensions. Note that the Ulrich complexity of such a cubic is strongly expected to be $9$. 

\begin{rem}
The same argument implies that there is no Ulrich sheaf of rank $6$ on a general cubic hypersurface of dimension $\ge 8$. Manivel found a family of smooth cubic eightfolds having an Ulrich sheaf of rank $9$ \cite[Proposition 2.2]{Man19}, namely, smooth linear sections of the Cartan cubic hypersurfaces. However, it is not clear that a general cubic eightfold can have an Ulrich sheaf of rank $9$. 
\end{rem}

To obtain an explicit presentation of such an Ulrich sheaf (equivalently, a matrix factorization), Shamash's construction plays a significant role throughout the rest of the paper. Let us briefly recall Shamash's construction. 
Let $Z$ be a subscheme contained in a hypersurface $X = V(f) \subset \mathbb{P}^{n}$ of degree $d$. Let $S=k[x_0, \cdots, x_{n}], R=S_X,$ and $S_Z$ be the coordinate rings of $\mathbb{P}^{n}, X, Z$ respectively. Let $F_{\bullet}$ be the minimal free $S$-resolution of $S_Z$. Since $Z \subset X$, we have a right exact sequence
\[
F_1 \otimes_S R \to F_0 \otimes_S R \simeq R \to S_Z \to 0,
\]
hence, there is a free $R$-resolution of $S_Z$ (possibly non-minimal)
\[
\cdots \to G_4 \oplus G_2(-d) \oplus G_0 (-2d) \to G_3 \oplus G_1(-d) \to G_2 \oplus G_0(-d) \to G_1 \to G_0 \to S_Z \to 0
\]
where $G_i = F_i \otimes_S R$. The resolution becomes eventually 2-periodic after a finite number of steps, and hence induces a matrix factorization of $f$. Such a matrix factorization provides a presentation matrix of an ACM sheaf on $X$ \cite[Corollary 6.3]{Eis80}. Since an Ulrich sheaf on $X$ corresponds to a matrix factorization $(A, B)$ of $f$ such that all the entries of $A$ are linear forms (and thus the entries of B are degree $(d-1)$-forms), one may obtain a presentation of an Ulrich sheaf when $S_Z$ has a pure resolution whose differentials have repeating degrees: $1$, $d-1$, then again by $1$, and so on. 

\begin{exmp}
A few easy examples can be easily found via Boij-S{\"o}derberg theory. Note that length 2 pure resolutions consist of degrees $0, 2, 3$ are multiples of 
\[
\begin{array}{ccc}
1 & - & - \\
- & 3 & 2 
\end{array}
\]
which is the Betti table of varieties of minimal degree of codimension $2$. For instance, a twisted cubic has the above Betti table. Since every smooth cubic surface in $\mathbb{P}^3$ contains such a twisted cubic, Shamash's construction provides a matrix factorization of the cubic surface by a $3 \times 3$ linear matrix. In other words, a smooth cubic surface always carries an Ulrich line bundle.

Let us consider the next case. Length 3 pure resolutions consists of degrees $0, 2, 3, 5$ are multiples of
\[
\begin{array}{cccc}
1 & - & - & - \\
- & 5 & 5 & - \\
- & - & - & 1
\end{array}
\]
which is the Betti table of del Pezzo varieties of codimension $3$ and degree $5$. For instance, a smooth cubic threefold contains an elliptic normal curve of degree 5, and a Pfaffian cubic fourfold contains a del Pezzo surface of degree 5. In both cases, Shamash's construction provides a $6 \times 6$ matrix which gives a matrix factorization of such a cubic hypersurface, and thus there is an Ulrich bundle of rank $2 = \frac{6}{3}$.
\end{exmp}

\begin{rem}
Shamash's construction possibly contains a cancellation. Let $C$ be a general 6-gonal curve of genus 10, and $D$ be a general $\mathfrak{g}_6^{1}$ on $C$. The linear system $|\omega_C(-D)|$ embeds $C$ into $\mathbb{P}^4$, with the following Betti table
\[
\begin{array}{cccc}
1 & - & - & - \\
- & - & - & - \\
- & 8 & 9 & - \\
- & - & - & 2
\end{array}
\]
(cf. Appendix of \cite{CH12} by Geiss and Schreyer). Applying Shamash's construction for a smooth cubic threefold $X$ containing $C$, we get the following non-minimal $R$-resolution 
\[
\cdots \to R(-7)^{\oplus 9} \oplus R(-6) \to R(-6)^{\oplus 10} \to R(-4)^{\oplus 9} \oplus R(-3) \to R(-3)^{\oplus 8} \to R \to S_C \to 0
\]
where $R$ is the homogeneous coordinate ring $S_X$ of $X$. Since the equation defining $X$ is contained in the $8$-dimensional vector space  $H^0(\mathcal I_C(3))$, a cancellation occurs: a few first terms of $R$-minimal resolution of $\mathcal S_C$ are indeed
\[
\cdots \to R(-6)^{\oplus 9} \to R(-4)^{\oplus 9} \to R(-3)^{\oplus 7} \to R \to S_C \to 0.
\]
This cancellation allows us to take a linear submatrix $R(-7)^{\oplus 9} \to R(-6)^{\oplus 9}$ which still induces a matrix factorization of $X$. As a consequence, we have a matrix factorization of $X$ by a $9 \times 9$ linear matrix which defines an Ulrich bundle of rank 3 on $X$.
\end{rem}

\section{Computing a matrix factorization}\label{Sect:MFCubic}

Using invariant theory, Manivel \cite[Corollary 2.3]{Man19} showed that the Cartan cubic hypersurface $\mathcal C \subset \mathbb{P}^{26}$ supports an Ulrich sheaf of rank $9$. By restricting onto a general linear section, we have a number of cubic hypersurfaces with a rank $9$ Ulrich sheaf. Since $\mathcal C$ is $E_6$-invariant, and the restriction map $\varPsi : Gr(9,27) // E_6 \dashrightarrow |\mathcal O_{\mathbb{P}^8}(3)| // PGL(8)$ is dominant \cite[Proposition 2.2]{IM14}, a general cubic sevenfold $X$ can be identified as a linear section of the Cartan cubic $\mathcal C \subset \mathbb{P}^{26}$. In particular, a general cubic sevenfold has a rank $9$ Ulrich sheaf. 

We first give an alternative proof of the existence of rank $9$ Ulrich sheaves on a general cubic sevenfold, as a quick application of Shamash's construction. 

\begin{thm}\label{Prop:Rk9UlrichExistence}
There is an Ulrich sheaf of rank $9$ on a general cubic hypersurface in $\mathbb{P}^{n}$ when $n \le 8$. 
\end{thm}

\begin{proof}
Since the restriction of an Ulrich sheaf onto a general hyperplane section is again Ulrich \cite[Lemma 2.4]{CH12}, it is enough to show that a general cubic sevenfold $X$ carries an Ulrich sheaf of rank 9. We first claim that $X$ contains a $3$-dimensional subscheme $Z$ having the following Betti table
\[
\begin{array}{cccccc}
1 & - & - & - & - & - \\
- & 10 & 16 & - & - & - \\
- & - & - & 16 & 10 & - \\
- & - & - & - & - & 1
\end{array}
\]
Note that the Betti table of a Mukai threefold of genus $7$ has the shape above, which can be obtained as a linear section of the spinor tenfold $\mathcal S_{10} \subset \mathbb{P}^{15}$. Also note that $X$ can be obtained as a linear section of the Cartan cubic $\mathcal C \subset \mathbb{P}^{26}$. Let $\Lambda \cong \mathbb{P}^8$ be a general linear subspace so that $X = \mathcal C \cap \Lambda$, and consider a general linear subspace of dimension $15$ containing $\Lambda$. By \cite[Lemma 5.2]{IM14}, any general $\mathbb{P}^{15} \subset \mathbb{P}^{26}$ is the linear span of some spinor tenfold contained in $\mathcal C$, so we denote this spinor tenfold again by $\mathcal S_{10} \subset \mathbb{P}^{15}$. Hence, a general cubic sevenfold $X$ contains a Mukai threefold $Z := \mathcal S_{10} \cap \Lambda \subset \mathcal C \cap \Lambda = X$ of genus $7$.

Let $S=k[x_0, \cdots, x_8], R=S_X$, and $S_Z$ be the coordinate rings of $\Lambda = \mathbb{P}^8, X, Z$, respectively. We apply Shamash's construction with $Z \subset X$. Note that the minimal free $S$-resolution of $S_Z$ is given by
\[
F_{\bullet} : 0 \to S(-8) \to S(-6)^{\oplus 10} \to S(-5)^{\oplus 16} \to S(-3)^{\oplus 16} \to S(-2)^{\oplus 10} \to  S \to S_Z \to 0.
\]
Let $G_i = F_i \otimes_S R$. We have an $R$-resolution of $S_Z$ of the following form:
\[
\cdots \to G_4  \oplus G_2 (-3) \oplus G_0 (-6) \to G_3 \oplus G_1(-3) \to G_2 \oplus G_0 (-3) \to G_1 \to G_0 \to S_Z \to 0.
\]
In particular, we have a linear map 
\[
d_6 : G_6 \oplus G_4(-3) \oplus G_2(-6) \oplus G_0(-9) \simeq R(-9)^{\oplus 27} \to G_5 \oplus G_3(-3) \oplus G_1 (-6) \simeq R(-8)^{\oplus 27}
\]
which forms a matrix factorization $(d_6, d_5)$ of $X$, \emph{i.e.}, $d_6$ gives a presentation of an Ulrich sheaf of rank $9$ on $X$.
\end{proof}

\begin{rem}
The same method also works for hypersurfaces of any degree. For instance, let $M$ be a generic skew-symmetric $7 \times 7$ matrix whose entries are linear forms, and let $Z$ be the variety generated by seven cubics which are 6-Pfaffians of $M$. Then $Z$ has the following Betti table
\[
\begin{array}{cccc}
1 & - & - & - \\
- & - & - & - \\
- & 7 & 7 & - \\
- & - & - & - \\
- & - & - & 1
\end{array}
\]
In particular, $Z$ has degree $14$ and codimension $3$. 

When we play this game with $5$ variables, $Z \subset \mathbb{P}^4$ will be a curve of degree $14$ and genus $15$ having the same Betti table above. By a deformation theoretic argument (cf. Appendix of \cite{CH12}), one can check that there is a dominating family $\mathfrak F$ of such curves, that is, the natural projection from the incidence scheme over the quartic threefolds $\{(C,X) \  | \ C \in \mathfrak F, C \subset X \in | \mathcal O_{\mathbb{P}^{4}}(4)|\} \to |\mathcal O_{\mathbb{P}^4}(4)| \simeq \mathbb{P}^{69}$ is dominant. For such a pair $(C,X)$, Shamash's construction provides an Ulrich sheaf of rank $2$ on $X$. In other words, a general quartic threefold in $\mathbb{P}^4$ is Pfaffian since it supports a rank $2$ Ulrich sheaf \cite[Proposition 8.5]{Bea00}. 

As the next case, let us consider a surface $Z \subset \mathbb{P}^5$ with the above Betti table. A computer-based computation claims that the incidence scheme of surfaces $Z \subset X$ contained in a quartic fourfold $X$ has local dimension $111$ at a randomly chosen point, hence cannot dominate the space of quartic fourfolds $|\mathcal O_{\mathbb{P}^5}(4)| \simeq \mathbb{P}^{125}$. Note that a general quartic fourfold is not Pfaffian \cite[Section 9]{Bea00}, and hence we are only able to obtain a smaller family of Pfaffian quartic fourfolds by this method.
\end{rem}

Two proofs for the existence of Ulrich sheaves of rank $9$ on a general cubic sevenfold look quite different at the first glance. Manivel's approach is based on the fact that a general cubic sevenfold is contained in a bigger variety (as a linear section) which equips with an Ulrich sheaf of small rank. On the other hand, our approach is based on the fact that a general cubic sevenfold contains a smaller variety satisfying special syzygy conditions. Nevertheless, it seems to be that both ideas are strongly related, for instance, the Cartan cubic $\mathcal C \subset \mathbb{P}^{26}$ appears as a key object in both approaches. Hence, it is natural to observe how both ideas intersect in the geometry of the Cartan cubic. 

For the rest of the paper, we describe a way how to recover the Cartan cubic and to compute its matrix factorization. This also provides an explicit matrix factorization of the Cartan cubic, and of a general cubic sevenfold by restrictions. Recall that the existence of a Mukai threefold $Z$ of genus $7$ in a general cubic sevenfold $X$ was crucial in Theorem \ref{Prop:Rk9UlrichExistence} and its proof. Note that a Mukai threefold of genus $7$ in $\mathbb{P}^8$ is a linear section of the spinor tenfold $\mathcal S_{10} \subset \mathbb{P}^{15}$. Hence, it is natural to consider a bigger cubic hypersurface containing the spinor tenfold $\mathcal S_{10}$, and to compute its matrix factorization by Shamash's construction. 

Before to proceed, we briefly recall how the Cartan cubic and the spinor tenfold are related. First note that the Lie group $E_6$ acts on the $27$-dimensional vector space $V_{27}$. After taking the projection, there are only three orbits in $\mathbb{P}^{26}$:
\begin{enumerate}
\item the Cayley plane $\mathbb{OP}^2 = E_6 / P_1$, the only closed orbit which is the Severi variety of dimension $16$;
\item $\mathcal C \setminus \mathbb{OP}^2$, where $\mathcal C \subset \mathbb{P}^{26}$ is the Cartan cubic;
\item $\mathbb{P}^{26} \setminus \mathcal C$, the dense open orbit.
\end{enumerate}
Note that the Levi factor of $P_1$ is isomorphic to $\mathbb{C}^{*} \times Spin_{10}$, and $T_{[e]} E_6/P_1$ identifies with the $16$-dimensional spinor representation. After the projectivization, there is only one closed orbit, which is the spinor variety $\mathcal S_{10} \subset \mathbb{P}^{15}$ (cf. \cite[Section 2]{Tev03}). Note also that the Cartan cubic is the secant variety of the Cayley plane, and the Cayley plane is the singular locus of the Cartan cubic.

On the other hand, our computations follow the converse direction: we start from the spinor tenfold $\mathcal S_{10} \subset \mathbb{P}^{15}$, and we describe how to construct a cubic hypersurface in $\mathbb{P}^{26}$ whose singular locus is the Cayley plane $\mathbb{OP}^2$. Note that the only cubic hypersurface satisfying such a ``secant--singular locus'' relation with the Cayley plane $\mathbb{OP}^2$ is the Cartan cubic $\mathcal C$. We used a computer algebra system \emph{Macaulay2} \cite{GS} to make computations appear in the rest of the paper. We refer to the appendix, together with \emph{Macaulay2} scripts, for readers who are interested in the very details of these computations.

Let $\mathcal S_{10} \subset \mathbb{P}^{15}=\mathbb{P}W$ be the spinor tenfold, where $W \simeq \bigwedge^0 k^5 \oplus \bigwedge^2 k^5 \oplus \bigwedge^4 k^5$. Let $x_0, x_{ij}, y_i=x_{\{1, \ldots, 5\} \setminus \{i \}}$ $(1 \le i< j \le 5)$ be the natural coordinates of $\mathbb{P}W$. Let $P$ be the skew-symmetric matrix
\[
P=\begin{pmatrix}
0 & x_{12} & x_{13} & x_{14} & x_{15} \\
-x_{12} & 0 & x_{23} & x_{24} & x_{25} \\
-x_{13} & -x_{23} & 0 & x_{34} & x_{35} \\
-x_{14} & -x_{24} & -x_{34} & 0 & x_{45} \\
-x_{15} & -x_{25} & -x_{35} & -x_{45} & 0 \\
\end{pmatrix}
.
\]
The following $10$ quadratic equations $q_1, \cdots, q_5, q_1^{\prime}, \cdots, q_5^{\prime}$ 
\begin{eqnarray*}
q_i & = & x_0y_i + (-1)^{i-1}  \pf(P,i) \\
q_i^{\prime} & = & \left(\begin{array}{ccccc} y_1 & y_2 & y_3 & y_4 & y_5 \end{array}\right) P_i
\end{eqnarray*}
generate $\mathcal S_{10}$, where $\pf(P,i)$ is the Pfaffian of the $(4 \times 4)$ matrix obtained by deleting the $i$-th row and column from $P$, and $P_i$ is the $i$-th column vector of $P$ \cite{Muk95}. As discussed above, the Betti table of $\mathcal S_{10}$ is
\[
\begin{array}{cccccc}
1 & - & - & - & - & - \\
- & 10 & 16 & - & - & - \\
- & - & - & 16 & 10 & - \\
- & - & - & - & - & 1
\end{array}
\]
so that $\mathcal S_{10}$ is arithmetically Gorenstein of degree $12$ and codimension $5$. Since we are interested in cubic hypersurfaces containing (a cone over) $\mathcal S_{10}$, we put $10$ extra variables $a_1, \cdots, a_5, b_1, \cdots, b_5$ of degree $1$, which correspond to $10$ quadric generators $q_i^{\prime}, q_i$ of $\mathcal S_{10}$. Consider the following universal cubic form
\[
F = \sum_{i=1}^{5} (q_i a_i + q_i^{\prime} b_i)
\]
defined in 26 variables $\mathbf{x,y,a,b}$. It is clear that the hypersurface $V(F) \subset \mathbb{P}^{25}$ contains a cone over $\mathcal S_{10}$, and hence Shamash's construction for this pair will provide a matrix factorization of $F$ by a $27 \times 27$ linear matrix $M_F$, as in the proof of Theorem \ref{Prop:Rk9UlrichExistence}. After taking suitable permutations of rows/columns, and multiplications on rows/columns by a nonzero scalar, which do not change the determinant up to constant multiples, we obtain the symmetric $(27 \times 27)$ matrix $M_F$ which induces a matrix factorization of $F$ as follows.
\[
\setlength{\arraycolsep}{1pt}{
\scalemath{0.6}{
{\left({\begin{array}{ccccccccccccccccccccccccccc}
       \cdot&\cdot&\cdot&\cdot&\cdot&\cdot&{-{y}_{2}}&{-{y}_{3}}&\cdot&{-{y}_{4}}&\cdot&\cdot&{-{y}_{5}}&\cdot&\cdot&\cdot&\cdot&\cdot&\cdot&\cdot&\cdot&\cdot&{-{x}_{12}}&{-{x}_{13}}&{-{x}_{14}}&{-{x}_{15}}&{-{a}_{1}}\\
       \cdot&\cdot&\cdot&\cdot&\cdot&\cdot&{y}_{1}&\cdot&{-{y}_{3}}&\cdot&{-{y}_{4}}&\cdot&\cdot&{-{y}_{5}}&\cdot&\cdot&\cdot&\cdot&\cdot&\cdot&\cdot&{x}_{12}&\cdot&{-{x}_{23}}&{-{x}_{24}}&{-{x}_{25}}&{-{a}_{2}}\\
       \cdot&\cdot&\cdot&\cdot&\cdot&\cdot&\cdot&{y}_{1}&{y}_{2}&\cdot&\cdot&{-{y}_{4}}&\cdot&\cdot&{-{y}_{5}}&\cdot&\cdot&\cdot&\cdot&\cdot&\cdot&{x}_{13}&{x}_{23}&\cdot&{-{x}_{34}}&{-{x}_{35}}&{-{a}_{3}}\\
       \cdot&\cdot&\cdot&\cdot&\cdot&\cdot&\cdot&\cdot&\cdot&{y}_{1}&{y}_{2}&{y}_{3}&\cdot&\cdot&\cdot&{-{y}_{5}}&\cdot&\cdot&\cdot&\cdot&\cdot&{x}_{14}&{x}_{24}&{x}_{34}&\cdot&{-{x}_{45}}&{-{a}_{4}}\\
       \cdot&\cdot&\cdot&\cdot&\cdot&\cdot&\cdot&\cdot&\cdot&\cdot&\cdot&\cdot&{y}_{1}&{y}_{2}&{y}_{3}&{y}_{4}&\cdot&\cdot&\cdot&\cdot&\cdot&{x}_{15}&{x}_{25}&{x}_{35}&{x}_{45}&\cdot&{-{a}_{5}}\\
       \cdot&\cdot&\cdot&\cdot&\cdot&\cdot&\cdot&\cdot&\cdot&\cdot&\cdot&\cdot&\cdot&\cdot&\cdot&\cdot&{y}_{1}&{y}_{2}&{y}_{3}&{y}_{4}&{y}_{5}&{a}_{1}&{a}_{2}&{a}_{3}&{a}_{4}&{a}_{5}&\cdot\\
       {-{y}_{2}}&{y}_{1}&\cdot&\cdot&\cdot&\cdot&\cdot&\cdot&\cdot&\cdot&\cdot&{-{a}_{5}}&\cdot&\cdot&{a}_{4}&{-{a}_{3}}&\cdot&\cdot&{-{x}_{45}}&{x}_{35}&{-{x}_{34}}&{b}_{2}&{-{b}_{1}}&\cdot&\cdot&\cdot&\cdot\\
       {-{y}_{3}}&\cdot&{y}_{1}&\cdot&\cdot&\cdot&\cdot&\cdot&\cdot&\cdot&{a}_{5}&\cdot&\cdot&{-{a}_{4}}&\cdot&{a}_{2}&\cdot&{x}_{45}&\cdot&{-{x}_{25}}&{x}_{24}&{b}_{3}&\cdot&{-{b}_{1}}&\cdot&\cdot&\cdot\\
       \cdot&{-{y}_{3}}&{y}_{2}&\cdot&\cdot&\cdot&\cdot&\cdot&\cdot&{-{a}_{5}}&\cdot&\cdot&{a}_{4}&\cdot&\cdot&{-{a}_{1}}&{-{x}_{45}}&\cdot&\cdot&{x}_{15}&{-{x}_{14}}&\cdot&{b}_{3}&{-{b}_{2}}&\cdot&\cdot&\cdot\\
       {-{y}_{4}}&\cdot&\cdot&{y}_{1}&\cdot&\cdot&\cdot&\cdot&{-{a}_{5}}&\cdot&\cdot&\cdot&\cdot&{a}_{3}&{-{a}_{2}}&\cdot&\cdot&{-{x}_{35}}&{x}_{25}&\cdot&{-{x}_{23}}&{b}_{4}&\cdot&\cdot&{-{b}_{1}}&\cdot&\cdot\\
       \cdot&{-{y}_{4}}&\cdot&{y}_{2}&\cdot&\cdot&\cdot&{a}_{5}&\cdot&\cdot&\cdot&\cdot&{-{a}_{3}}&\cdot&{a}_{1}&\cdot&{x}_{35}&\cdot&{-{x}_{15}}&\cdot&{x}_{13}&\cdot&{b}_{4}&\cdot&{-{b}_{2}}&\cdot&\cdot\\
       \cdot&\cdot&{-{y}_{4}}&{y}_{3}&\cdot&\cdot&{-{a}_{5}}&\cdot&\cdot&\cdot&\cdot&\cdot&{a}_{2}&{-{a}_{1}}&\cdot&\cdot&{-{x}_{25}}&{x}_{15}&\cdot&\cdot&{-{x}_{12}}&\cdot&\cdot&{b}_{4}&{-{b}_{3}}&\cdot&\cdot\\
       {-{y}_{5}}&\cdot&\cdot&\cdot&{y}_{1}&\cdot&\cdot&\cdot&{a}_{4}&\cdot&{-{a}_{3}}&{a}_{2}&\cdot&\cdot&\cdot&\cdot&\cdot&{x}_{34}&{-{x}_{24}}&{x}_{23}&\cdot&{b}_{5}&\cdot&\cdot&\cdot&{-{b}_{1}}&\cdot\\
       \cdot&{-{y}_{5}}&\cdot&\cdot&{y}_{2}&\cdot&\cdot&{-{a}_{4}}&\cdot&{a}_{3}&\cdot&{-{a}_{1}}&\cdot&\cdot&\cdot&\cdot&{-{x}_{34}}&\cdot&{x}_{14}&{-{x}_{13}}&\cdot&\cdot&{b}_{5}&\cdot&\cdot&{-{b}_{2}}&\cdot\\
       \cdot&\cdot&{-{y}_{5}}&\cdot&{y}_{3}&\cdot&{a}_{4}&\cdot&\cdot&{-{a}_{2}}&{a}_{1}&\cdot&\cdot&\cdot&\cdot&\cdot&{x}_{24}&{-{x}_{14}}&\cdot&{x}_{12}&\cdot&\cdot&\cdot&{b}_{5}&\cdot&{-{b}_{3}}&\cdot\\
       \cdot&\cdot&\cdot&{-{y}_{5}}&{y}_{4}&\cdot&{-{a}_{3}}&{a}_{2}&{-{a}_{1}}&\cdot&\cdot&\cdot&\cdot&\cdot&\cdot&\cdot&{-{x}_{23}}&{x}_{13}&{-{x}_{12}}&\cdot&\cdot&\cdot&\cdot&\cdot&{b}_{5}&{-{b}_{4}}&\cdot\\
       \cdot&\cdot&\cdot&\cdot&\cdot&{y}_{1}&\cdot&\cdot&{-{x}_{45}}&\cdot&{x}_{35}&{-{x}_{25}}&\cdot&{-{x}_{34}}&{x}_{24}&{-{x}_{23}}&\cdot&\cdot&\cdot&\cdot&\cdot&{x}_{0}&\cdot&\cdot&\cdot&\cdot&{-{b}_{1}}\\
       \cdot&\cdot&\cdot&\cdot&\cdot&{y}_{2}&\cdot&{x}_{45}&\cdot&{-{x}_{35}}&\cdot&{x}_{15}&{x}_{34}&\cdot&{-{x}_{14}}&{x}_{13}&\cdot&\cdot&\cdot&\cdot&\cdot&\cdot&{x}_{0}&\cdot&\cdot&\cdot&{-{b}_{2}}\\
       \cdot&\cdot&\cdot&\cdot&\cdot&{y}_{3}&{-{x}_{45}}&\cdot&\cdot&{x}_{25}&{-{x}_{15}}&\cdot&{-{x}_{24}}&{x}_{14}&\cdot&{-{x}_{12}}&\cdot&\cdot&\cdot&\cdot&\cdot&\cdot&\cdot&{x}_{0}&\cdot&\cdot&{-{b}_{3}}\\
       \cdot&\cdot&\cdot&\cdot&\cdot&{y}_{4}&{x}_{35}&{-{x}_{25}}&{x}_{15}&\cdot&\cdot&\cdot&{x}_{23}&{-{x}_{13}}&{x}_{12}&\cdot&\cdot&\cdot&\cdot&\cdot&\cdot&\cdot&\cdot&\cdot&{x}_{0}&\cdot&{-{b}_{4}}\\
       \cdot&\cdot&\cdot&\cdot&\cdot&{y}_{5}&{-{x}_{34}}&{x}_{24}&{-{x}_{14}}&{-{x}_{23}}&{x}_{13}&{-{x}_{12}}&\cdot&\cdot&\cdot&\cdot&\cdot&\cdot&\cdot&\cdot&\cdot&\cdot&\cdot&\cdot&\cdot&{x}_{0}&{-{b}_{5}}\\
       \cdot&{x}_{12}&{x}_{13}&{x}_{14}&{x}_{15}&{a}_{1}&{b}_{2}&{b}_{3}&\cdot&{b}_{4}&\cdot&\cdot&{b}_{5}&\cdot&\cdot&\cdot&{x}_{0}&\cdot&\cdot&\cdot&\cdot&\cdot&\cdot&\cdot&\cdot&\cdot&\cdot\\
       {-{x}_{12}}&\cdot&{x}_{23}&{x}_{24}&{x}_{25}&{a}_{2}&{-{b}_{1}}&\cdot&{b}_{3}&\cdot&{b}_{4}&\cdot&\cdot&{b}_{5}&\cdot&\cdot&\cdot&{x}_{0}&\cdot&\cdot&\cdot&\cdot&\cdot&\cdot&\cdot&\cdot&\cdot\\
       {-{x}_{13}}&{-{x}_{23}}&\cdot&{x}_{34}&{x}_{35}&{a}_{3}&\cdot&{-{b}_{1}}&{-{b}_{2}}&\cdot&\cdot&{b}_{4}&\cdot&\cdot&{b}_{5}&\cdot&\cdot&\cdot&{x}_{0}&\cdot&\cdot&\cdot&\cdot&\cdot&\cdot&\cdot&\cdot\\
       {-{x}_{14}}&{-{x}_{24}}&{-{x}_{34}}&\cdot&{x}_{45}&{a}_{4}&\cdot&\cdot&\cdot&{-{b}_{1}}&{-{b}_{2}}&{-{b}_{3}}&\cdot&\cdot&\cdot&{b}_{5}&\cdot&\cdot&\cdot&{x}_{0}&\cdot&\cdot&\cdot&\cdot&\cdot&\cdot&\cdot\\
       {-{x}_{15}}&{-{x}_{25}}&{-{x}_{35}}&{-{x}_{45}}&\cdot&{a}_{5}&\cdot&\cdot&\cdot&\cdot&\cdot&\cdot&{-{b}_{1}}&{-{b}_{2}}&{-{b}_{3}}&{-{b}_{4}}&\cdot&\cdot&\cdot&\cdot&{x}_{0}&\cdot&\cdot&\cdot&\cdot&\cdot&\cdot\\
       {-{a}_{1}}&{-{a}_{2}}&{-{a}_{3}}&{-{a}_{4}}&{-{a}_{5}}&\cdot&\cdot&\cdot&\cdot&\cdot&\cdot&\cdot&\cdot&\cdot&\cdot&\cdot&{-{b}_{1}}&{-{b}_{2}}&{-{b}_{3}}&{-{b}_{4}}&{-{b}_{5}}&\cdot&\cdot&\cdot&\cdot&\cdot&\cdot\\
       \end{array}}\right)}}
}
\]

The cubic form $F$ does not define the Cartan cubic $\mathcal C$ at the moment, since the minimal irreducible representation $V_{27}$ of $E_6$ is $27$-dimensional whereas there are only $26$ variables in our coordinate ring. To study the differences between the coordinates of $V_{27}$, we take the Jacobian ideal $J(F)$ and its generators, namely, $26$ partial derivatives of $F$ which define the singular locus $\sing(V(F))$ of $V(F) \subset \mathbb{P}^{25}$. The Betti table of $\sing (V(F))$ is given by
\[
\setlength{\arraycolsep}{3pt}{
\scalemath{0.9}{
\begin{array}{ccccccccccccccccc}
1&-&-&-&-&-&-&-&-&-&-&-&-&-&-&-&-\\
-&26&62&-&-&-&-&-&-&-&-&-&-&-&-&-&-\\
-&-&120&911&2470&4719&8008&11440&12870&11440&8008&4368&1820&560&120&16&1\\
-&-&-&-&-&351&650&351&-&-&-&-&-&-&-&-&-\\
-&-&-&-&-&-&-&-&78&27&-&-&-&-&-&-&-\\
-&-&-&-&-&-&-&-&-&-&1&-&-&-&-&-&-
\end{array}
}
}
\],
whereas the Betti table of the Cayley plane $\mathbb{OP}^2  = \sing (\mathcal C) \subset \mathbb{P}^{26}$ is given by
\[
\setlength{\arraycolsep}{3pt}{
\scalemath{0.9}{
\begin{array}{ccccccccccc}
1&-&-&-&-&-&-&-&-&-&-\\
-&27&78&-&-&-&-&-&-&-&-\\
-&-&-&351&650&351&-&-&-&-&-\\
-&-&-&-&-&351&650&351&-&-&-\\
-&-&-&-&-&-&-&-&78&27&-\\
-&-&-&-&-&-&-&-&-&-&1
\end{array}
}
}
\]
We may wildly guess that both singular loci are closely related, due to similar shapes of their Betti tables. Indeed, $\sing(V(F))$ contains an embedded component $\Lambda \cong \mathbb{P}^9$, a linear subspace in $\mathbb{P}^{25}$ cut out by the 16 variables $\mathbf{x, y}$ which are coordinates for $\mathcal{S}_{10} \subset \mathbb{P}^{15}$. The difference of two Betti tables comes from $\Lambda$, so that the Koszul relations of $\Lambda$ appear in the quadratic strand (= $3$rd row) of the table. After detaching $\Lambda$ from $\sing (V(F))$, one can check that the Betti table of the subscheme defined by the ideal $(J(F):I(\Lambda))$ has exactly the same shape as the one of $\mathbb{OP}^2 \subset \mathbb{P}^{26}$. In particular, the ideal $(J(F):I(\Lambda))$ contains $1$ more independent quadric which does not appear as quadric generators of $J(F)$. We compute this ``missing quadric'', which is given by 
\[
a_1b_1 + a_2b_2 + a_3b_3 + a_4b_4 + a_5b_5 \in (J(F):I(\Lambda))_2.
\]
Since $J(F)$ does not contain such an element, it is natural to enlarge the Jacobian ideal so that a new ideal should contain the above quadric. We take  the simplest way; put one more variable $w$, and define  $F_{\mathcal C} := F - w \cdot \left( \sum_{i=1}^5 a_i b_i \right)$ so that the above quadric appears as $- \frac{\partial F_{\mathcal C}}{\partial w}$. Note that the hypersurface $V(F)$ is a (special) linear section of this new cubic hypersurface $V(F_{\mathcal C})$, which might behave much nicer since the  Jacobian ideal $J(F_{\mathcal C})$ no more contains an embedded component.

Let us have a closer analysis on the $27$ quadrics obtained as partial derivatives of $F_{\mathcal C}$. Note first that partial derivatives of $F_{\mathcal C}$ with respect to the variables $a_i, b_i$ are $q_i( \mathbf{x, y}) - b_i w$ and $q_i^{\prime}(\mathbf{x,y}) - a_i w$, where $q_i, q_i^{\prime}$ are the quadric generators of the spinor variety $\mathcal S_{10}$. Since each $a_i$ corresponds to $q_i^{\prime}$ (and $b_i$ corresponds to $q_i$), they can be understood as the homogenizations of these correspondences. Also note that $16$ partial derivatives with respect to $x_0, x_{ij}, y_i$ correspond to a linear syzygy among the quadrics $q_1, \cdots, q_5, q_1^{\prime}, \cdots, q_5^{\prime}$. For instance, a linear syzygy $x_0q_1^{\prime} + x_{12}q_2+x_{13}q_3+x_{14}q_4+x_{15}q_5=0$ corresponds to the quadric
\[
x_0a_1 + x_{12}b_2+x_{13}b_3+x_{14}b_4+x_{15}b_5 = \frac{\partial F_{\mathcal C}}{\partial y_1},
\]
via substituting $q_i$ by $b_i$, and $q_i^{\prime}$ by $a_i$. Hence, it sounds natural that the partial derivatives of $F_{\mathcal C}$ are ``relations''  of total degree $2$ which are derived from $\mathcal S_{10}$. Finally, via the backward-substitution $a_i \mapsto q_i^{\prime}$ and $b_i \mapsto q_i$, the quadric $\frac{\partial F_{\mathcal C}}{\partial{w}} = - \sum_{i=1}^5 a_i b_i$ corresponds to $-\sum_{i=1}^5 q_i q_i^{\prime}$, which is identically zero since $P$ is skew-symmetric. 

To sum up, we lead to the description of the hypersurfaces $V(F), V(F_{\mathcal C})$, and their singular loci as follows. 

\begin{thm}\label{prop:RecoverCartanCubic}
Let $F_{\mathcal C} := F - \left(\sum_{i=1}^5 a_i b_i \right) w$ as above. The hypersurface $V(F_{\mathcal C})$ is the Cartan cubic hypersurface $\mathcal C$ in $\mathbb{P}^{26}$, and the hypersurface $V(F) \subset V(w) = \mathbb{P}^{25}$ is its hyperplane section. The singular locus of $V(F)$ is the union of a hyperplane section of the Cayley plane $\mathbb{OP}^2 \subset \mathbb{P}^{26}$ and an embedded component $\Lambda \cong \mathbb{P}^9$.
\end{thm}

\begin{proof}
Let $T=k[x_0, x_{ij}, y_i, a_i, b_i, w ; \ 1 \le i < j \le 5]$ be the polynomial ring in $27$ variables. We are interested in the role of these $27$ coordinates. We take the partial derivatives of $F_{\mathcal C}$, and observe the Jacobian ideal $J(F_{\mathcal C})$. First of all, $16$ of them obtained by taking partial derivatives with respect to $x_0, x_{ij}, y_i$ correspond to the 16 linear syzygies among the $10$ quadrics $q_1, \cdots, q_5, q_1^{\prime}, \cdots, q_5^{\prime}$ as described above. Next, $10$ of them obtained by taking partial derivatives with respect to $a_i, b_i$ are $q_i - b_i w, q_i^{\prime}-a_i w$. Finally, the partial derivative with respect to $w$ gives the last generator
\[
-(a_1 b_1 + a_2 b_2 + a_3 b_3 + a_4 b_4 + a_5 b_5) = \frac{\partial F_{\mathcal C}}{\partial w},
\] 
which corresponds to the quadratic relation $-\sum_{i=1}^5 q_i q_i^{\prime} = 0$. 

In particular, the singular locus of $V(F_\mathcal C)$ coincides with the closure of the image of $\mathbb{P}^{16}$ under the following rational map defined by the linear system of quadrics containing (the cone over) the spinor tenfold $\mathcal S_{10} \subset \mathbb{P}^{15}$ 
\begin{eqnarray*}
\mathbb{P}^{16} & \dashrightarrow & \mathbb{P}^{26} \\
(\xi,Z)=[ X_0 : X_{ij} : Y_i : Z ] & \mapsto & [X_0 Z : X_{ij} Z : Y_i Z : q_i^{\prime}(\xi) : q_{i}(\xi) : Z^2 ]
\end{eqnarray*}
which is well-defined outside of the spinor tenfold $\mathcal S_{10} = V(q_i, q_i^{\prime}, Z) \subset V(Z) = \mathbb{P}^{15}$. Such a variety must be the Severi variety $\mathbb{OP}^2 \subset \mathbb{P}^{26}$ of dimension $16$ \cite[Theorem 4.5]{Zak93}, and its secant variety is the cubic hypersurface $V(F_{\mathcal C})$, thus we conclude that the hypersurface $V(F_{\mathcal C})$ exactly coincides with the Cartan cubic $\mathcal C$.

The last statement follows from a simple computation. Since the Jacobian $J(F)$ is radical, it is straightforward that
\[
V(J(F_{\mathcal C}), w) \subset V(J( F_{\mathcal C}|_{(w=0)})) = V(J(F)) \subset V(w) = \mathbb{P}^{25},
\]
and the closure of the set difference is given by the ideal quotient
\[
(J(F) : J(F_{\mathcal C})|_{(w=0)}) = (x_0, x_{ij}, y_i )
\]
which defines an embedded component $\Lambda \cong \mathbb{P}^9 \subset V(w)=\mathbb{P}^{25}$ as desired.
\end{proof}
\begin{rem}\label{Rmk:CorrespondenceBetween27Lines}
Indeed, there is a beautiful correspondence between these 27 variables and the 27 lines on a smooth cubic, namely,
\[
\begin{array}{ccc}
x_0 & \mapsto & \{6 \} \\
b_i & \mapsto & \{ i \} \\
x_{ij} & \mapsto & \{i, j \} \\
a_i & \mapsto & \{i, 6 \} \\
y_i & \mapsto & \{i\}^c = \{1, \cdots , 6 \} \setminus \{ i \} \\
w & \mapsto & \{6\}^c = \{1, 2, 3, 4, 5 \} .
\end{array}
\]
Via this identification, one can check that $F_{\mathcal C}$ is a (signed) sum of $45$ cubic monomials corresponding to $45$ tritangent planes $(\{i\},\{j\}^c, \{i,j\}), (\{i_1, j_1\},\{i_2,j_2\},\{i_3,j_3\})$ of 27 lines (cf. \cite[Chapter 9]{Dol12}). This gives an alternative proof that our cubic $F_{\mathcal C}$ coincides with the Cartan cubic form, without passing through analysis of its singular locus. It is well known that the automorphism group of the Cartan cubic $\mathcal C = V(F_{\mathcal C})$ is $E_6$. 
\end{rem}

Note that there are 3 types for $27$ lines:
\begin{enumerate}
\item $\{i\} ~ (1 \le i \le 6)$, corresponds to the exceptional line $E_i$;
\item $\{i, j \} ~ (1 \le i < j \le 6)$, corresponds to the line $L-E_i-E_j$;
\item $\{j\}^c ~ (1 \le j \le 6)$, corresponds to the line $2L-\sum_{i=1}^6 E_i + E_j$.
\end{enumerate}
Since the Picard group of a smooth cubic surface is isomorphic to $\mathbb{Z}^7$, we give the $\mathbb{Z}^7$-grading on the $27$ variables in a natural way. Note that the above $45$ tritangent planes correspond to the triple of lines whose sum have multidegree $\{3,-1,-1,-1,-1,-1,-1\}$, that is, the multidegree of the anticanonical divisor $-K$ for a smooth cubic.

For convenience, we re-order the variables with respect to the $\mathbb{Z}^7$-grading, namely, into the following order:
\[
(b_1, \cdots, b_5, x_0, x_{12}, \cdots, x_{45}, a_1, \cdots, a_5, y_1, \cdots, y_5,w).
\]

Note that the Hessian matrix of the Cartan cubic $F_{\mathcal C}$ induces a matrix factorization of itself:
\[
\setlength{\arraycolsep}{1pt}{
\scalemath{0.6}{
{\left({\begin{array}{ccccccccccccccccccccccccccc}
       \cdot&\cdot&\cdot&\cdot&\cdot&\cdot&{-{y}_{2}}&{-{y}_{3}}&\cdot&{-{y}_{4}}&\cdot&\cdot&{-{y}_{5}}&\cdot&\cdot&\cdot&{-w}&\cdot&\cdot&\cdot&\cdot&\cdot&{-{x}_{12}}&{-{x}_{13}}&{-{x}_{14}}&{-{x}_{15}}&{-{a}_{1}}\\
       \cdot&\cdot&\cdot&\cdot&\cdot&\cdot&{y}_{1}&\cdot&{-{y}_{3}}&\cdot&{-{y}_{4}}&\cdot&\cdot&{-{y}_{5}}&\cdot&\cdot&\cdot&{-w}&\cdot&\cdot&\cdot&{x}_{12}&\cdot&{-{x}_{23}}&{-{x}_{24}}&{-{x}_{25}}&{-{a}_{2}}\\
       \cdot&\cdot&\cdot&\cdot&\cdot&\cdot&\cdot&{y}_{1}&{y}_{2}&\cdot&\cdot&{-{y}_{4}}&\cdot&\cdot&{-{y}_{5}}&\cdot&\cdot&\cdot&{-w}&\cdot&\cdot&{x}_{13}&{x}_{23}&\cdot&{-{x}_{34}}&{-{x}_{35}}&{-{a}_{3}}\\
       \cdot&\cdot&\cdot&\cdot&\cdot&\cdot&\cdot&\cdot&\cdot&{y}_{1}&{y}_{2}&{y}_{3}&\cdot&\cdot&\cdot&{-{y}_{5}}&\cdot&\cdot&\cdot&{-w}&\cdot&{x}_{14}&{x}_{24}&{x}_{34}&\cdot&{-{x}_{45}}&{-{a}_{4}}\\
       \cdot&\cdot&\cdot&\cdot&\cdot&\cdot&\cdot&\cdot&\cdot&\cdot&\cdot&\cdot&{y}_{1}&{y}_{2}&{y}_{3}&{y}_{4}&\cdot&\cdot&\cdot&\cdot&{-w}&{x}_{15}&{x}_{25}&{x}_{35}&{x}_{45}&\cdot&{-{a}_{5}}\\
       \cdot&\cdot&\cdot&\cdot&\cdot&\cdot&\cdot&\cdot&\cdot&\cdot&\cdot&\cdot&\cdot&\cdot&\cdot&\cdot&{y}_{1}&{y}_{2}&{y}_{3}&{y}_{4}&{y}_{5}&{a}_{1}&{a}_{2}&{a}_{3}&{a}_{4}&{a}_{5}&\cdot\\
       {-{y}_{2}}&{y}_{1}&\cdot&\cdot&\cdot&\cdot&\cdot&\cdot&\cdot&\cdot&\cdot&{-{a}_{5}}&\cdot&\cdot&{a}_{4}&{-{a}_{3}}&\cdot&\cdot&{-{x}_{45}}&{x}_{35}&{-{x}_{34}}&{b}_{2}&{-{b}_{1}}&\cdot&\cdot&\cdot&\cdot\\
       {-{y}_{3}}&\cdot&{y}_{1}&\cdot&\cdot&\cdot&\cdot&\cdot&\cdot&\cdot&{a}_{5}&\cdot&\cdot&{-{a}_{4}}&\cdot&{a}_{2}&\cdot&{x}_{45}&\cdot&{-{x}_{25}}&{x}_{24}&{b}_{3}&\cdot&{-{b}_{1}}&\cdot&\cdot&\cdot\\
       \cdot&{-{y}_{3}}&{y}_{2}&\cdot&\cdot&\cdot&\cdot&\cdot&\cdot&{-{a}_{5}}&\cdot&\cdot&{a}_{4}&\cdot&\cdot&{-{a}_{1}}&{-{x}_{45}}&\cdot&\cdot&{x}_{15}&{-{x}_{14}}&\cdot&{b}_{3}&{-{b}_{2}}&\cdot&\cdot&\cdot\\
       {-{y}_{4}}&\cdot&\cdot&{y}_{1}&\cdot&\cdot&\cdot&\cdot&{-{a}_{5}}&\cdot&\cdot&\cdot&\cdot&{a}_{3}&{-{a}_{2}}&\cdot&\cdot&{-{x}_{35}}&{x}_{25}&\cdot&{-{x}_{23}}&{b}_{4}&\cdot&\cdot&{-{b}_{1}}&\cdot&\cdot\\
       \cdot&{-{y}_{4}}&\cdot&{y}_{2}&\cdot&\cdot&\cdot&{a}_{5}&\cdot&\cdot&\cdot&\cdot&{-{a}_{3}}&\cdot&{a}_{1}&\cdot&{x}_{35}&\cdot&{-{x}_{15}}&\cdot&{x}_{13}&\cdot&{b}_{4}&\cdot&{-{b}_{2}}&\cdot&\cdot\\
       \cdot&\cdot&{-{y}_{4}}&{y}_{3}&\cdot&\cdot&{-{a}_{5}}&\cdot&\cdot&\cdot&\cdot&\cdot&{a}_{2}&{-{a}_{1}}&\cdot&\cdot&{-{x}_{25}}&{x}_{15}&\cdot&\cdot&{-{x}_{12}}&\cdot&\cdot&{b}_{4}&{-{b}_{3}}&\cdot&\cdot\\
       {-{y}_{5}}&\cdot&\cdot&\cdot&{y}_{1}&\cdot&\cdot&\cdot&{a}_{4}&\cdot&{-{a}_{3}}&{a}_{2}&\cdot&\cdot&\cdot&\cdot&\cdot&{x}_{34}&{-{x}_{24}}&{x}_{23}&\cdot&{b}_{5}&\cdot&\cdot&\cdot&{-{b}_{1}}&\cdot\\
       \cdot&{-{y}_{5}}&\cdot&\cdot&{y}_{2}&\cdot&\cdot&{-{a}_{4}}&\cdot&{a}_{3}&\cdot&{-{a}_{1}}&\cdot&\cdot&\cdot&\cdot&{-{x}_{34}}&\cdot&{x}_{14}&{-{x}_{13}}&\cdot&\cdot&{b}_{5}&\cdot&\cdot&{-{b}_{2}}&\cdot\\
       \cdot&\cdot&{-{y}_{5}}&\cdot&{y}_{3}&\cdot&{a}_{4}&\cdot&\cdot&{-{a}_{2}}&{a}_{1}&\cdot&\cdot&\cdot&\cdot&\cdot&{x}_{24}&{-{x}_{14}}&\cdot&{x}_{12}&\cdot&\cdot&\cdot&{b}_{5}&\cdot&{-{b}_{3}}&\cdot\\
       \cdot&\cdot&\cdot&{-{y}_{5}}&{y}_{4}&\cdot&{-{a}_{3}}&{a}_{2}&{-{a}_{1}}&\cdot&\cdot&\cdot&\cdot&\cdot&\cdot&\cdot&{-{x}_{23}}&{x}_{13}&{-{x}_{12}}&\cdot&\cdot&\cdot&\cdot&\cdot&{b}_{5}&{-{b}_{4}}&\cdot\\
       {-w}&\cdot&\cdot&\cdot&\cdot&{y}_{1}&\cdot&\cdot&{-{x}_{45}}&\cdot&{x}_{35}&{-{x}_{25}}&\cdot&{-{x}_{34}}&{x}_{24}&{-{x}_{23}}&\cdot&\cdot&\cdot&\cdot&\cdot&x_0&\cdot&\cdot&\cdot&\cdot&{-{b}_{1}}\\
       \cdot&{-w}&\cdot&\cdot&\cdot&{y}_{2}&\cdot&{x}_{45}&\cdot&{-{x}_{35}}&\cdot&{x}_{15}&{x}_{34}&\cdot&{-{x}_{14}}&{x}_{13}&\cdot&\cdot&\cdot&\cdot&\cdot&\cdot&x_0&\cdot&\cdot&\cdot&{-{b}_{2}}\\
       \cdot&\cdot&{-w}&\cdot&\cdot&{y}_{3}&{-{x}_{45}}&\cdot&\cdot&{x}_{25}&{-{x}_{15}}&\cdot&{-{x}_{24}}&{x}_{14}&\cdot&{-{x}_{12}}&\cdot&\cdot&\cdot&\cdot&\cdot&\cdot&\cdot&x_0&\cdot&\cdot&{-{b}_{3}}\\
       \cdot&\cdot&\cdot&{-w}&\cdot&{y}_{4}&{x}_{35}&{-{x}_{25}}&{x}_{15}&\cdot&\cdot&\cdot&{x}_{23}&{-{x}_{13}}&{x}_{12}&\cdot&\cdot&\cdot&\cdot&\cdot&\cdot&\cdot&\cdot&\cdot&x_0&\cdot&{-{b}_{4}}\\
       \cdot&\cdot&\cdot&\cdot&{-w}&{y}_{5}&{-{x}_{34}}&{x}_{24}&{-{x}_{14}}&{-{x}_{23}}&{x}_{13}&{-{x}_{12}}&\cdot&\cdot&\cdot&\cdot&\cdot&\cdot&\cdot&\cdot&\cdot&\cdot&\cdot&\cdot&\cdot&x_0&{-{b}_{5}}\\
       \cdot&{x}_{12}&{x}_{13}&{x}_{14}&{x}_{15}&{a}_{1}&{b}_{2}&{b}_{3}&\cdot&{b}_{4}&\cdot&\cdot&{b}_{5}&\cdot&\cdot&\cdot&x_0&\cdot&\cdot&\cdot&\cdot&\cdot&\cdot&\cdot&\cdot&\cdot&\cdot\\
       {-{x}_{12}}&\cdot&{x}_{23}&{x}_{24}&{x}_{25}&{a}_{2}&{-{b}_{1}}&\cdot&{b}_{3}&\cdot&{b}_{4}&\cdot&\cdot&{b}_{5}&\cdot&\cdot&\cdot&x_0&\cdot&\cdot&\cdot&\cdot&\cdot&\cdot&\cdot&\cdot&\cdot\\
       {-{x}_{13}}&{-{x}_{23}}&\cdot&{x}_{34}&{x}_{35}&{a}_{3}&\cdot&{-{b}_{1}}&{-{b}_{2}}&\cdot&\cdot&{b}_{4}&\cdot&\cdot&{b}_{5}&\cdot&\cdot&\cdot&x_0&\cdot&\cdot&\cdot&\cdot&\cdot&\cdot&\cdot&\cdot\\
       {-{x}_{14}}&{-{x}_{24}}&{-{x}_{34}}&\cdot&{x}_{45}&{a}_{4}&\cdot&\cdot&\cdot&{-{b}_{1}}&{-{b}_{2}}&{-{b}_{3}}&\cdot&\cdot&\cdot&{b}_{5}&\cdot&\cdot&\cdot&x_0&\cdot&\cdot&\cdot&\cdot&\cdot&\cdot&\cdot\\
       {-{x}_{15}}&{-{x}_{25}}&{-{x}_{35}}&{-{x}_{45}}&\cdot&{a}_{5}&\cdot&\cdot&\cdot&\cdot&\cdot&\cdot&{-{b}_{1}}&{-{b}_{2}}&{-{b}_{3}}&{-{b}_{4}}&\cdot&\cdot&\cdot&\cdot&x_0&\cdot&\cdot&\cdot&\cdot&\cdot&\cdot\\
       {-{a}_{1}}&{-{a}_{2}}&{-{a}_{3}}&{-{a}_{4}}&{-{a}_{5}}&\cdot&\cdot&\cdot&\cdot&\cdot&\cdot&\cdot&\cdot&\cdot&\cdot&\cdot&{-{b}_{1}}&{-{b}_{2}}&{-{b}_{3}}&{-{b}_{4}}&{-{b}_{5}}&\cdot&\cdot&\cdot&\cdot&\cdot&\cdot\\
       \end{array}}\right)}
}
}
\]
We denote this Hessian matrix of $F_{\mathcal C}$ by $\mathcal H(F_{\mathcal C})$. It is composed of block matrices, having a number of symmetries. For instance, the block at the lower-left (and also the upper-right) corner is a generic $6 \times 6$ skew-symmetric matrix (this is the reason why we subtract the term $w \left( \sum_{i=1}^5 a_i b_i \right)$ from $F$, which provides the correct signed sum). Note that the block in the middle is a symmetric $15 \times 15$  matrix, which coincide with the Hessian matrix of the Pfaffian of the upper-right generic $6 \times 6$ skew-symmetric matrix in $15$ variables $\{x_{ij}, a_i\}$. Also note that this Pfaffian defines a secant variety of the Severi variety $Gr(2,6) \subset \mathbb{P}^{14}$, and its Hessian (= the middle block of our matrix) is a matrix factorization of this cubic. See also \cite{Kim19} for more examples and classification of such cubics. 

 When we restrict $\mathcal H(F_{\mathcal C})$ on the hyperplane $V(w)$, two matrices $M_F$ and $\mathcal H(F_{\mathcal C})|_{w=0}$ exactly coincide. In particular, two matrices only differ by $10$ entries containing $-w$.

\begin{rem}
It is not very surprising that the Hessian matrix $\mathcal H(F_{\mathcal C})$ of the Cartan cubic form $F_{\mathcal C}$ induces a matrix factorization of itself. The group $E_6$ acts on the $27$-dimensional vector space $V_{27}$ so that the Cartan cubic form $F_{\mathcal C}$ is the unique irreducible invariant (up to constant multiples), hence a theorem of Ein and Shepherd-Barron \cite[Theorem 2.8]{ESB89} implies that there is a co-coordinate system $(s_0, \cdots, s_{26})$ on $V_{27}$ such that the gradient map
\[
\begin{array}{cccc}
\nabla F_{\mathcal C} : & \mathbb{P}^{26} & \dashrightarrow & \mathbb{P}^{26} \\
 &(s_0, \cdots, s_{26}) &  \mapsto & \left( t_0 = \frac{\partial F_{\mathcal C}}{\partial s_0}, \cdots, t_{26} = \frac{\partial F_{\mathcal C}}{\partial s_{26}}\right)
\end{array}
\]
is a Cremona involution, that is, $(\nabla F_{\mathcal C})^2 = id$. 

In fact, our coordinates $\mathbf{x,y,a,b},w$ is already normalized in this viewpoint; one can check that the gradient map $\nabla F_{\mathcal C}$ is a Cremona involution with respect to $\mathbf{x,y,a,b}, w$. For convenience, let us denote $s_0, \cdots, s_{26}$ for the coordinates $\mathbf{x,y,a,b},w$, and let $t_i$ be the partial derivative with respect to $s_i$. 

Then, $F_{\mathcal C}$ satisfies 
\begin{eqnarray*}
F_{\mathcal C}(t_0, \cdots, t_{26})  & = &  F_{\mathcal C}(s_0, \cdots, s_{26})^2, \\
\frac{\partial F_{\mathcal C}(t_0, \cdots, t_{26})}{\partial t_i} & = & s_i F_{\mathcal C}(s_0, \cdots, s_{26}).
\end{eqnarray*}
Applying the Euler formula to the second equality, we have
\[
\frac{\partial^2 F_{\mathcal C}(t_0, \cdots, t_{26})}{\partial t_i \partial t_j} = \left[\frac{1}{2} s_i s_j +  F_{\mathcal C}(s_0, \cdots, s_{26} ) \frac{\partial s_i}{\partial t_j} \right].
\]
The left-hand-side is a linear form in $t_0, \cdots, t_{26}$, hence, it is a quadratic form in $s_0, \cdots, s_{26}$ via the substitution $t_i = \frac{\partial F_{\mathcal C}}{\partial s_i}$. In particular, the term $q_{ij} := F_{\mathcal C} (s_0, \cdots, s_{26}) \frac{\partial s_i}{\partial t_j}$ is  quadratic in $s_0, \cdots, s_{26}$ for each $i, j$. Let $\mathcal Q = (q_{ij})$ be the $27 \times 27$ matrix composed of these quadrics. Since the Hessian matrix $\mathcal H(F_{\mathcal C}) = \left( \frac{\partial t_i}{\partial s_j} \right)$ and $\mathcal Q$ satisfy $\mathcal H(F_{\mathcal C}) \mathcal Q = \mathcal Q \mathcal H(F_{\mathcal C}) = F_{\mathcal C} \cdot Id$, we conclude that $(\mathcal H(F_{\mathcal C}), \mathcal Q)$ is a matrix factorization of $F_{\mathcal C}$ by its Hessian. We refer to  \cite[Proposition 2.5]{IM14} and \cite[Example 2.1.8]{Abu18} for another explanation using the exceptional Jordan algebra of $\mathbb{O}$-Hermitian matrices. 
\end{rem}
 
\begin{ack}
The authors thank David Eisenbud for helpful discussions. This work was supported by Project I.6 of SFB-TRR 195 ``Symbolic Tools in Mathematics and their Application'' of the German Research Foundation (DFG). 
\end{ack}


\def\cprime{$'$} \def\cprime{$'$} \def\cprime{$'$} \def\cprime{$'$}
  \def\cprime{$'$} \def\cprime{$'$} \def\dbar{\leavevmode\hbox to
  0pt{\hskip.2ex \accent"16\hss}d} \def\cprime{$'$} \def\cprime{$'$}
  \def\polhk#1{\setbox0=\hbox{#1}{\ooalign{\hidewidth
  \lower1.5ex\hbox{`}\hidewidth\crcr\unhbox0}}} \def\cprime{$'$}
  \def\cprime{$'$} \def\cprime{$'$} \def\cprime{$'$}
  \def\polhk#1{\setbox0=\hbox{#1}{\ooalign{\hidewidth
  \lower1.5ex\hbox{`}\hidewidth\crcr\unhbox0}}} \def\cdprime{$''$}
  \def\cprime{$'$} \def\cprime{$'$} \def\cprime{$'$} \def\cprime{$'$}
\providecommand{\bysame}{\leavevmode\hbox to3em{\hrulefill}\thinspace}
\providecommand{\MR}{\relax\ifhmode\unskip\space\fi MR }
\providecommand{\MRhref}[2]{%
  \href{http://www.ams.org/mathscinet-getitem?mr=#1}{#2}
}
\providecommand{\href}[2]{#2}

\vskip1cm

\section{Appendix: computer-based computations with \emph{Macaulay2} scripts}

We address \emph{Macaulay2} \cite{GS} scripts which we used throughout the paper with a few comments.

First, we define the spinor tenfold $\mathcal S_{10} \subset \mathbb{P}^{15}$. Since the generic $5 \times 5$ skew-symmetric linear matrix induce the quadric generators of $\mathcal S_{10}$, we set up as follows.

\begin{verbatim}
i1 : kk=QQ;
x=symbol x;
S=kk[x_0,apply(subsets(toList(1..5),2),ij->x_(10*ij_0+ij_1)),y_1..y_5];

X=matrix{{0,x_12,x_13,x_14,x_15},{0,0,x_23,x_24,x_25},
{0,0,0,x_34,x_35},{0,0,0,0,x_45},{0,0,0,0,0}};
X=map(S^5,,X-transpose X);
pf=mingens pfaffians(4,X);
A=matrix{apply(5,i->x_0*y_(i+1)-(-1)^i*pf_(0,i))};
B=matrix{apply(5,i->y_(i+1))}*X;
spin=ideal A+ideal B;
\end{verbatim}

Note that the entries of $A$, $B$ are generators of the ideal ``spin'' which defines $\mathcal S_{10}$. It is well-known that $\mathcal S_{10}$ has the desired Betti table to obtain an Ulrich sheaf of rank $9$ via Shamash's construction.

\begin{verbatim}
i2 : fspin=res spin; betti fspin

            0  1  2  3  4 5
o2 = total: 1 10 16 16 10 1
         0: 1  .  .  .  . .
         1: . 10 16  .  . .
         2: .  .  . 16 10 .
         3: .  .  .  .  . 1
\end{verbatim}

We put extra variables $a_1, \cdots, a_5, b_1, \cdots, b_5$ which corresponds to those $10$ quadric generators $q_1, \cdots, q_5, q_1^{\prime}, \cdots , q_5^{\prime}$, and consider the cubic $F = \sum a_i q_i + b_i q_i^{\prime}$. We apply Shamash's construction for the hypersurface defined by $F$ which contains a cone over $\mathcal S_{10}$. 

\begin{verbatim}
i3 : SExt=kk[gens S, a_1..a_5,b_1..b_5];
aa=matrix{apply(5,i->a_(i+1))};
bb=matrix{apply(5,i->b_(i+1))};
F=aa*sub(transpose A,SExt)+bb*sub(transpose B,SExt);
R=SExt/ideal F;
spinR=sub(spin,R);
fperiodic=res(spinR,LengthLimit=>6);
betti fperiodic

            0  1  2  3  4  5  6
o3 = total: 1 10 17 26 27 27 27
         0: 1  .  .  .  .  .  .
         1: . 10 17  .  .  .  .
         2: .  .  . 26 27  .  .
         3: .  .  .  .  . 27 27
\end{verbatim}

We check that the resolution over the hypersurface ring becomes $2$-periodic after $4$ steps, and the first linear matrix appears as $d_6$. Indeed, this gives a matrix factorization of $F$.

\begin{verbatim}
i4 : M=fperiodic.dd_6;
ann coker sub(M,SExt)==ideal F

o4 = true
\end{verbatim}

Due to computational issues (e.g. choice of basis), it is hard to observe that it coincides with the restriction of the Hessian matrix $\mathcal H(F_{\mathcal C})$ as mentioned above. We need to manipulate the matrix by a number of certain permutations of rows/columns, and a number of multiples by nonzero constants on rows/columns to obtain the matrix we seen above. To reduce the steps, we give a multigrade on each variable, and compute the same matrix over a multigraded polynomial ring as follows. 

First of all, we compute all the possible multigrading structures on the variables defining $F$ so that $F$ becomes homogeneous:
\begin{verbatim}
i5 : cubics=(entries(coefficients F)_0)_0;
varSExt=(entries vars SExt)_0;
incMatrix=matrix apply(cubics, t->(
   apply(varSExt,l->(if codim ideal (t,l)==1 then 1 else 0))));
rel=id_(ZZ^(#cubics-1))||matrix {apply(#cubics-1, i->(-1))}; 
stdGrading=matrix apply(#varSExt,i->{1});
possibleGradings=(gens ker((transpose rel)*incMatrix));
gradingLLL=LLL mingens image(stdGrading|(possibleGradings%stdGrading));
rank gradingLLL

o5 = 7
\end{verbatim}

Hence, the $26$ variables $\mathbf{x,y,a,b}$ admit a $\mathbb{Z}^7$-grading. Let us compare with the $\mathbb{Z}^7$-grading for lines on a smooth cubic surface, which is described in Remark $3.4$. Note that we drop the last multidegree $\{2,-1,-1,-1,-1,-1,0\}$ (corresponding to $w$) at the moment.

\begin{verbatim}
i6 : Es=apply(6,i->apply(7,j->if i==j-1 then 1 else 0));
Fij=apply(subsets(toList(0..5),2),ij->{1,0,0,0,0,0,0}-Es_(ij_0)-Es_(ij_1));
Gs=apply(6,i->{2,-1,-1,-1,-1,-1,-1}+Es_i);
grading=Es|Fij|Gs;
permGrading=(matrix grading)^{5,6..15,21..25,16..20,0..4};
image permGrading==image gradingLLL

o6 = true
\end{verbatim}

Hence, both gradings are equivalent, and hence we can plug in the $\mathbb{Z}^7$-grading as in Remark $3.4$. To fit with the above matrices, we select a certain permutation of rows/columns carefully:

\begin{verbatim}
i7 : 
loadPackage ("K3Carpets",Reload=>true)
Sall=kk[gens SExt,Degrees=>entries permGrading]
Fall=map(Sall^1,,substitute(F,Sall));
Rall=Sall/ideal sub(F,Sall);
FperAll=allGradings(fperiodic,Rall);
M=map(Sall^(-degrees FperAll_5),Sall^(-degrees FperAll_6),
   sub(FperAll.dd_6,Sall));

degsTargetM=degrees target M;
degsSourceM=degrees source M;

varOrder={5,4,3,2,1,0,6,7,11,8,12,15,9,13,
   16,18,10,14,17,19,20,26,25,24,23,22,21};
sortedTargetDegs=(sort degsTargetM)_varOrder;
sortedSourceDegs=(reverse sort degsSourceM)_varOrder;

blocks=apply(sortedTargetDegs,d->apply(sortedSourceDegs,e->(
	    L1= select(rank source M,i->degsSourceM_i==e);
	    L2= select(rank target M,i->degsTargetM_i==d);
	    M^L2_L1)));
netList blocks;
netList (BS=apply(blocks,b->apply(b,m->sub(m,Sall))));

M=map(Sall^0,Sall^27,0);
for i from 0 to 26 do (
    N=map(target BS_i_0,Sall^0,0);
    for j from 0 to 26 do (N=N|BS_i_j);
    M=M||N)
M
\end{verbatim}

The matrix $M$ looks much better, in particular, the basis are now well-placed with respect to the given multigrading structure. However, still there are issues on the choice of coefficients (even $M$ is not symmetric at the moment), and hence we need to take certain multiples on rows/columns.

Before a further correction to $M$, we first check how the universal cubic $F$ apart from the Cartan cubic $F_{\mathcal C}$. First note that the singular locus of $V(F)$, which is generated by $26$ quadrics, contains an embedded component $\Lambda$. When we take it off, then the remaining set is generated by $27$ quadrics, so there is one more quadric. We compute this extra quadric, which is $\sum_{i=1}^5 a_i b_i$:
\begin{verbatim}
i8 : IF=ideal F;
JacIF=saturate ideal jacobian IF;
embComponent=sub(ideal(gens S), SExt);
(numgens (JacIF:embComponent), numgens JacIF)==(27,26)

o8 = true

i9: extraQuadric=flatten entries(generators(JacIF:embComponent));
for i from 0 to 25 do extraQuadric=
   delete((flatten entries generators JacIF)_i,extraQuadric);
extraQuadric

o9 = {a b  + a b  + a b  + a b  + a b }
       1 1    2 2    3 3    4 4    5 5
\end{verbatim}

We compute the Betti table of the singular locus of $V(F)$, with and without the embedded component $\Lambda$. The computational cost is a bit high, we reduce it on a finite field and use the ``minimalBetti'' command. As result, we compute the Betti tables discussed in Section $3$.

\begin{verbatim}
i10 : p=nextPrime(10^3);
Sfin=ZZ/p[gens SExt];

time minimalBetti sub(JacIF,Sfin)
time minimalBetti sub((JacIF:embComponent),Sfin)
\end{verbatim}

It is natural to adjust $F$ slightly by putting a further extra variable ``$w$'' so that the partial derivative with respect to $w$ corresponds to this quadric. The result gives the Cartan cubic hypersurface in $\mathbb{P}^{26}$. 

\begin{verbatim}
i11 : T=kk[gens SExt, w];
FC=sub((flatten entries F)_0,T)-sum(apply(5, i->a_(i+1)*b_(i+1)*w));
IFC=ideal FC;
JacIFC=saturate ideal jacobian IFC;
\end{verbatim}

One can check that the singular locus of the hypersurface defined by $F$ is the union of the hyperplane section of the singular locus of the Cartan cubic (= Severi variety of dimension 16) and a linear subspace $\Lambda$ as an embedded component:

\begin{verbatim}
i12 : (JacIF : sub(JacIFC,SExt)) == embComponent

o12 = true
\end{verbatim}

We compute the Hessian matrix of the Cartan cubic, following the same order on variables as in Remark $3.4$.
\begin{verbatim}
i13 : permVarT=matrix({{b_1..b_5,x_0, x_12,x_13,x_23,x_14,x_24,x_34,x_15,
   x_25,x_35,x_45,a_1..a_5,y_1..y_5,w}});
HFC = (diff(permVarT,transpose diff(permVarT, FC)))
\end{verbatim}

Finally, we correct the matrix factorization $M$ of $F$ computed above, by solving the equation $D M = \mathcal H(F_{\mathcal C})|_{w=0} E$, where $D, E$ are diagonal matrices. As a result, we obtain a symmetric matrix $M^{\prime} = DME^{-1}$ which coincides with the restriction of the Hessian matrix $\mathcal H({F_{\mathcal C}})|_{w=0}$ ($=$ the matrix $M_F$ in Section $3$).

\begin{verbatim}
i14 : resHFC=sub(HFC,Sall);
SallDE=Sall[d_1..d_27,e_1..e_27];
listD=toList(d_1..d_27);
listE=toList(e_1..e_27);
D=diagonalMatrix(listD);
E=diagonalMatrix(listE);

relationsDE=unique flatten(apply(flatten entries vars SExt,i->unique(
   flatten entries diff(sub(i,SallDE),sub(D*M-resHFC*E, SallDE)))));
relDE=matrix(apply(relationsDE,i->{diff(vars SallDE,i)}));
solDE=(syz relDE);

solD=sub(sub(D,(transpose solDE)),Sall);
solE=sub(sub(E,(transpose solDE)),Sall);

Mprime=solD*M*inverse(solE);
(Mprime==transpose Mprime,Mprime==resHFC,ann coker Mprime==ideal Fall)

o14 = (true, true, true)

i15 : Mprime
\end{verbatim}

\vspace{1cm}

\end{document}